\documentclass{amsart}

\usepackage[foot]{amsaddr}
\usepackage{todonotes}
\usepackage{graphicx}
\usepackage{hyperref}
\usepackage{standalone}
\usepackage[LGR,T1]{fontenc}
\usepackage[greek,english]{babel}
\languageattribute{greek}{polutoniko} 

\newtheorem{thm}{Theorem}[section]
\newtheorem{coro}[thm]{Corollary}
\newtheorem{prop}[thm]{Proposition}
\newtheorem{lem}[thm]{Lemma}
\theoremstyle{definition}
\newtheorem{defi}[thm]{Definition}
\newtheorem{rmk}[thm]{Remark}
\newtheorem{ex}[thm]{Example}

\newtheorem*{notation*}{Notation}

\newcommand{\Z}{\mathbb Z}
\newcommand{\N}{\mathbb N}
\newcommand{\R}{\mathbb R}

\newcommand{\mcg}{\mathrm{MCG}(S)}
\newcommand{\PSL}{\mathrm{PSL}(2,\mathbb{C})}
\newcommand{\HH}{\mathbb{H}}
\newcommand{\scc}{\mathcal{S}}

\newcommand{\T}{\mathcal{T}(S)}

\newcommand{\TT}{\mathcal{T}(S)\times\mathcal{T}(S)}
\newcommand{\curr}{\mathrm{Curr}(S)}

\newcommand{\Vol}{\mathrm{vol}}

\newcommand{\pmf}{\mathcal{PMF}(S)}
\newcommand{\mf}{\mathcal{MF}(S)}
\newcommand{\ext}{\mathrm{Ext}}

\newcommand{\dt}{d_{\mathcal{T}}}

\newcommand{\Lipi}{\mathrm{Lip}^{1}_{b}}

\title{A horofunction counterpart to Teichm\"uller distance}
\author{Hidetoshi Masai}
\address{Humanities and Sciences/Museum Careers, Musashino Art University
Bldg. 12, 1-736 Ogawa-cho, Kodaira-shi, Tokyo 187-8505}
\email{hmasai@musabi.ac.jp}
\thanks{The work of the author was partially supported by JSPS KAKENHI Grant Number 23K03085.}

\begin{document}
\begin{abstract}
We generalize the horofunction compactification to maps that are not distance functions. As an application we define a horofunction counterpart to the Teichm\"uller distance, and discuss its properties.
\end{abstract}
\maketitle
\tableofcontents
\section{Introduction}
Teichm\"uller space admits several natural distances, including the Teichm\"uller distance,
Thurston's asymmetric distance and the Weil-Petersson distance, etc.
A feature that is shared by some of these distances is that they are constructed from natural functionals, 
such as extremal length, hyperbolic length, or renormalized volumes of quasi-Fuchsian manifolds.
Therefore, it is natural to ask what kind of distances can be obtained in a similar way.
In this paper, we show that a ``horofunction counterpart'' to the Teichm\"uller distance can be obtained in such a way, and discuss its properties.

Horofunction compactifications have found numerous applications in the study of 
Teichm\"uller space and beyond. 
Walsh~\cite{Walsh} proved that the horofunction boundary of Thurston's distance recovers the Thurston boundary, while Liu--Su~\cite{LS} showed that the horofunction boundary of the Teichm\"uller distance coincides with the Gardiner--Masur boundary.
Horofunction techniques also play an important role in stochastic analysis of mapping class groups, see e.g. \cite{Karlsson-Two, KL, Masai, MT}.

The main new results of this paper can be summarized as follows.
\begin{itemize}
  \item A generalized construction of horofunctions is developed for maps of the form
        $I:M\times N\to\mathbb{R}$, encompassing classical examples 
        such as the Teichm\"uller and Thurston distances (Theorem \ref{thm.summary}).
  \item It is shown  (Theorem \ref{thm.volcc}) that a distance defined via convex core volume yields translation lengths 
        equal to the volumes of mapping tori of pseudo-Anosov maps, 
        complementing the analogous statement for the renormalized volume given in \cite{Masai}.
   \item We introduce a new distance $d_{E,2}$ (see Definition \ref{defi.E12}), which arises naturally from extremal length and should be viewed as the horofunction counterpart to the Teichm\"uller distance. We show that it is within bounded distance of Thurston's distance on the thick part of Teichm\"uller space (Theorem \ref{thm.E2}). We also prove that there are natural ``horo-compactifications'' via extremal length, which include the Gardiner--Masur and Thurston boundaries (Theorem \ref{thm.ext-final}).
\end{itemize}

\subsection{Organization of the paper}
The horofunction boundary of a metric space $(M,d)$ is classically defined 
as the closure of the functions
\begin{equation}
M \ni z \mapsto d(\cdot,z)-d(b,z) \in \mathrm{Lip}_1^{d}(M,d), \label{eq.horo-ori}
\end{equation}
where $b \in M$ is a basepoint and $\mathrm{Lip}_1^{d}(M,d)$ denotes the space of 1-Lipschitz functions that vanish at $b\in M$.
In Sections \ref{sec.dist} and \ref{sec.horoLip}, we ask whether this standard construction 
is the only possible way to obtain horofunction boundaries. 
More precisely, given a map $I:M\times N\to \mathbb{R}$, we show that one can associate 
distances on $M$ and $N$, and then construct corresponding horofunction compactifications. 
This general method not only recovers the setting of \eqref{eq.horo-ori} 
when $M=N$ and $I=d$, but also includes natural examples such as 
length functions on Teichm\"uller space.

In Section \ref{sec.Tdists}, we discuss how several distances on Teichm\"uller space 
arise in the framework developed in Sections \ref{sec.dist} and \ref{sec.horoLip}. 
Furthermore, in Subsection \ref{sec.volumes}, we discuss the renormalized volume 
and the convex core volume of quasi-Fuchsian manifolds and its relation to the volume of mapping tori.

A key ingredient in the definition of $d_{E,2}$ is the work of 
Mart\'inez--Granado and Thurston~\cite{GT} (see Sections \ref{sec.curr} and \ref{sec.dist-len}). 
Their result extends many natural length functions (including extremal length), 
originally defined on closed curves of $S$, to the space of geodesic currents $\curr$ 
introduced by Bonahon~\cite{Bonahon}. 
Bonahon proved that assigning the so-called Liouville current $L_{X}$ to 
$X\in\T$ yields an embedding $\T \hookrightarrow \curr$. 
It is known that the infinitesimal behavior of the intersection number $i(L_X,L_Y)$ 
recovers the Weil--Petersson distance~\cite[Theorem~19]{Bonahon}. 
In contrast, we give a global inequality relating $i(L_X,L_Y)$ to Thurston's Lipschitz distance, with explicit dependence on the systole in Proposition~\ref{prop.int-est}.

In Sections \ref{sec.curr} and \ref{sec.dist-len}, we observe that several 
properties of the intersection number $i(L_X,L_Y)$ have analogues for the 
extremal length $\ext_{X}(L_{Y})$ of the Liouville current given by \cite{GT}.
For this purpose, the extension of Minsky's inequality to geodesic currents 
plays a key role (Lemma \ref{lem.Minsky}). 
We also show that some of Miyachi's results on extremal length~\cite{ray2} 
can be extended to geodesic currents. 
After establishing several properties of $d_{E,2}$, we prove an extremal length 
analogue of Proposition~\ref{prop.int-est} in Corollary \ref{coro.Ext-est}.

\section{Distances and Lipschitz maps}\label{sec.dist}
Let $C(X)$ denote the set of real-valued functions on $X$.
Two functions $f,g\in C(X)$ are {\em equivalent} if there exists a constant $c\in\R$ such that $f(x) - g(x) = c$ for all $x\in X$.
We denote by $C_*(X)$ the equivalence classes of functions on $X$.

Let $M, N$ be sets and $I: M\times N\to\R$ a function. 
First, we define a notion that ensures $I$ is ``non-degenerate".
\begin{defi}\label{defi.sp}
Given $x\in M$, let $\ell_x:N\to \R $ denote the function $I(x,\cdot)$.
A map $I:M\times N\to\R$ {\em 	separates points of $M$} if for any distinct pair of points
$x, y\in M$, we have 
\begin{enumerate}
\item there exists $z\in N$ such that $I(x,z)-I(y,z)>0$, and 
\item $\ell_x\neq\ell_y$ in $C_*(N)$.
\end{enumerate}
We define analogously what it means for $I:M\times N\to\R$ to separate points of $N$.
\end{defi}
\begin{rmk}
To define a distance, the property (1) suffices.
We require (2) to consider horofunctions and compactifications.
\end{rmk}

We define a ($L^{\infty}$-)distance via the map $I$.
\begin{defi}
Let $M, N$ be sets and $I: M\times N\to\R$ a function.
Define a function $d_{I,M}:M\times M\to \R_{\geq 0}\cup\{\infty\}$ by
$$d_{I,M}(x,y) = \sup_{z\in N} \{I(x,z)-I(y,z)\}.$$
\end{defi}

The map $d_{I,M}$ often becomes asymmetric.
\begin{defi}
Given an asymmetric metric space $(M,d)$, we define a symmetrization of $d$ by 
$$
d^{\mathrm{sym}}(x,y) :=\max\{d(x,y), d(y,x)\}.
$$
This symmetrization will be used throughout the paper.
\end{defi}

We prepare a notion that we need to construct a horofunction boundary via $I$.
\begin{defi}
	Let $K>0$, and suppose that $(M,d)$ is a possibly asymmetric metric space.
	A map $I:M\times N\to \R $ is said to satisfy the {\em $K$-Lipschitz Condition} if for any $x,y\in M$ and $z\in N$,
	we have 
	$$|I(x,z)-I(y,z)|\leq K\cdot d^{\mathrm{sym}}(x,y).$$
\end{defi}

\begin{thm}\label{thm.dist}
	Suppose that $I:M\times N\to \R $ separates points on $M$.
	Suppose further that $$d_{I,M}(\cdot,\cdot)<\infty.$$
	Then, the function $d_{I,M}:M\times M\to \R_{\geq 0}\cup\{\infty\}$ is a (possibly asymmetric) distance on $M$.
	Furthermore, we have that
	\begin{itemize}
		\item 	$I$ satisfies $1$-Lipschitz condition with respect to $d_{I,M}$.
		\item        In particular, for any $n\in N$, $I(\cdot,n)$ is continuous for the topology determined by $d_{I,M}^{\mathrm{sym}}$.

	\end{itemize}

\end{thm}
\begin{proof}

	First, we prove that $d_{I,M}$ separates points and satisfies the triangle inequality.
	Note that $d_{I,M}(x,x) = 0$.
	Now suppose that $d_{I,M}(x,y) = 0$. 
	Then we have $$\sup_{z\in N} \{I(x,z) - I(y,z)\} = 0$$ for any $z\in N$,
	which implies $x = y$ as $I$ separates points of $M$.
	
	
	Take $\varepsilon>0$ arbitrarily.
	Then there exists $z_\varepsilon$ such that 
	$$d_{I,M}(x,y) \leq I(x,z_\varepsilon) - I(y,z_\varepsilon)+\varepsilon.$$
	Hence we have for any $a\in M$,
	\begin{align*}
		d_{I,M}(x,y) &= \sup_{z\in N}\{I(x,z)-I(y,z)\}\\
		&\leq I(x,z_\varepsilon) - I(y,z_\varepsilon)+\varepsilon\\
		&= I(x,z_\varepsilon) - I(a,z_\varepsilon) + 
		I(a,z_\varepsilon) -I(y,z_\varepsilon)+\varepsilon\\
		&\leq \sup_{z\in N}\{I(x,z)-I(a,z)\} + \sup_{z\in N}\{I(a,z)-I(y,z)\}+ \varepsilon\\ 
		&\leq d_{I,M}(x,a) + d_{I,M}(a,y) + \varepsilon.
	\end{align*}
	Since $\varepsilon>0$ was arbitrary, we have $d_{I,M}(x,y)\leq d_{I,M}(x,a)+d_{I,M}(a,y)$.
	Therefore, $d_{I,M}$ satisfies the triangle inequality.
	
	By definitions of $d_{I,M}$ and $I(\cdot,\cdot)$,
	 $I(\cdot,n)$ satisfies $1$-Lipschitz condition which implies that $I(\cdot, n)$ is continuous in the topology determined by $d_{I,M}^{\mathrm{sym}}$.

\end{proof}

\begin{rmk}\label{rmk.tri-d}
By the definition of $d_{I,M}$, we have
\begin{align*}
I(x,y)+I(y,z)-I(x,z)\geq I(x,y)-d_{I,M}(x,y).
\end{align*}
The left-hand side is nonnegative if $I$ satisfies the triangle inequality, and the right-hand side is always nonpositive.
Hence, one observes that the difference  $I(x, y) - d_{I,M}(x, y)$  measures how far  $I$  deviates from satisfying the triangle inequality.

In particular, if $I:M\times M\to \R_{\geq 0}$ is a distance (in this case we need to have $N=M$), we see that 
$$I(x, y) - d_{I,M}(x, y) = 0,$$
namely $I = d_{I,M}$.
\end{rmk}

\subsection{Examples of distances}
\begin{ex}[Quotients]
	Given sets $M,N$ and a function $I:M\times N\to \R$, let us suppose that 
	$d_{I,M}(x,y)<\infty$ for any $x,y\in M$. Then one can define an equivalence relation on $M$: let $x,y\in M$. We define
	$$x\sim y\iff \exists C\in\R \text{ s.t. }I(x,z) - I(y,z)=C,~\forall z\in N.$$
	Then $I$ naturally descends to $\bar I:(M/\!\!\sim)\times N\to\R$.
	If we further have (1) of Definition \ref{defi.sp} then $\bar I$ separates points of $M/\!\!\sim$.
\end{ex}

\begin{ex}[Euclidean distance]
Let $\langle \cdot, \cdot\rangle:\R^{n}\times\R^{n}\to \R$ denote the standard inner product.
We define $I:\R^{n}\times\R^{n}\setminus\{0\}\to\R$ by
$$I(x,y) = \left\langle x, \frac{y}{\|y\|}\right\rangle.$$
Then the distance $d_{I,\R^{n}}$ coincides with the standard Euclidean distance.
\end{ex}

\begin{ex}[Hyperbolic distance]
In \cite[Lemma 2.2]{Minsky-product}, Minsky observed that the function
$$
I:\HH\times\R\to\R, \quad ((x+y\sqrt{-1}), t)\mapsto \log\left(y+\frac{(t+x)^{2}}{y}\right)
$$
gives the hyperbolic distance on $\HH$ as $d_{I,\HH}$.
\end{ex}

\begin{ex}[Funk Geometry]
Given a convex subset of a Euclidean space, one may consider the so-called Funk geometry:

The \textit{Funk metric}, denoted by $F_{\Omega}$, 
on a convex domain $\Omega$ is defined for $x$ and $y$ in $\Omega$ by $F_{\Omega}(x,x) = 0$ and by 
$$
F_{\Omega}(x, y) = \sup_{H} \log \left( \frac{d(x,H)}{d(y,H)} \right),
$$
where $H$ runs over all the supporting hyperplanes.
See for example \cite[Chapter 2, Corollary 2.7]{handbookHilbert}.
Hence, the Funk metric is obtained by
$$
I:\Omega\times\text{space of supporting hyperplane}\to\R, \quad I(x,H) = \log(d(x,H)).
$$
See e.g. \cite{handbookHilbert} for more details about Funk, and Hilbert geometry.
\end{ex}

\begin{ex}[Spectral distances]
Suppose that $M$ is a set of length distances on a surface $S$,
and $N$ is a set of curves on $S$.
Then we define $I(x,n):=\log \ell_{x}(n)$, the length of the curve $n\in N$ with respect to the distance $x\in M$. Then we have
$$
d_{I,M}(x,y) = \sup_{n\in N}\log\frac{\ell_{x}(n)}{\ell_{y}(n)}.
$$
If $I$ separates points of $M$ and $d_{I,M}$ is finite, then $d_{I,M}$ is called a {\em spectral distance}.
Such distances appear in many situations, and their relation to machine learning is discussed in \cite{Karlsson-machine}.

In the latter half of this paper, we especially focus on the case where $M$ is the Teichm\"uller space of a surface. For instance, the Teichm\"uller distance and 
the Thurston distance appear in this way.
\end{ex}


\subsection{Metric completions}
When $d_{I,M}$ becomes an incomplete metric, we may consider the metric completion via Cauchy sequences.
One observes that $I$ extends to its metric completion.
\begin{prop}\label{prop.incomplete}
Let $I:M\times N\to\R$ be a function which separates points of $M$, and $d_{I,M}$ denotes the distance obtained by Theorem \ref{thm.dist}.
Suppose that $$d_{I,M}^{\mathrm{sym}}:M\times M\to \R_{\geq 0}$$ is incomplete and 
$\overline M$ denotes its metric completion.
Then the map $I:M\times N\to \R$ extends as $\bar I:\overline M\times N\to \R$.
\end{prop}
\begin{proof}
The metric completion is obtained as the space of equivalence classes of Cauchy sequences of the metric.
Let $\{x_{k}\}\subset M$ be a Cauchy sequence of $d_{I,M}^{\mathrm{sym}}$.
Then by the definition of $d_{I,M}$, we see that for any $n\in N$.
$\{I(x_{k},n)\}\subset \R$ is a Cauchy sequence in $\R$ with the standard metric.
Hence, we may define 
$$I(\{x_{k}\},n):=\lim_{k\to\infty}I(x_{k},n)$$
for any $n\in N$.
\end{proof}
\begin{ex}
As we will see later in Section \ref{sec.volumes}, we may define several distances on the Teichm\"uller space via volume functions.
The resulting distances are known to be incomplete, and Proposition \ref{prop.incomplete} shows that those volume functions can be extended to the completions.
\end{ex}

\section{Horofunctions via a Lipschitz map}\label{sec.horoLip}
Let $I: M\times N\to\R$ be a function that separates points of $M$ and $N$.

We first consider horofunctions defined by $I$.
Let us fix a basepoint $b\in M$.
\begin{defi}\label{defi.horo}
Given $z\in N$, the horofunction $\ell_z:M\to \R$ at $z$ with respect to $I$ is defined by
$$\ell_z(\cdot) = I(\cdot,z) - I(b,z).$$
\end{defi}
\begin{rmk}
In the previous section, we considered $\ell_{z}(\cdot)\in C_{*}(M)$.
In Definition \ref{defi.horo}, we fix one of its representatives.
\end{rmk}

\begin{defi}
	Let $\Lipi(M,d)$ denote the space of $1$-Lipschitz maps with respect to the distance $d$ on $M$:
	$$\Lipi(M,d):=\{f:M\to\R\mid f(b) = 0, ~f\text{ satisfies $1$-Lipschitz condition}\}.$$	
\end{defi}
We equip $\Lipi(M,d)$ with the topology of pointwise convergence.
This topology is equivalent to the topology of uniform convergence
on compact sets, which is also equivalent to the compact-open topology.
\begin{prop}
	We have $\ell_z\in \Lipi(M,d_{I,M}^{\mathrm{sym}})$.
\end{prop}
\begin{proof}
	Since $$|\ell_z(x)-\ell_z(y)| = |I(x,z)-I(y,z)|\leq d_{I,M}^{\mathrm{sym}}(x,y),$$
	 $\ell_z$ is $1$-Lipschitz. That $\ell_z(b)=0$ is obvious.
\end{proof}

\begin{prop}[see e.g. {\cite[Proposition 3.1]{MT}}]\label{prop.Lip-compact}
Let $(M,d)$ be a separable metric space. Then the space $\Lipi(M,d)$ of $1$-Lipschitz maps is a compact Hausdorff second countable (hence metrizable) space.
\end{prop}

Gromov first considered horofunctions with respect to distances
 to compactify metric spaces.
It turns out that we may embed the space $N$ into $\Lipi(M)$ via $I$.
\begin{thm}\label{thm.lipi}
	Let $I: M\times N\to \R$ separate points of $M$ and $N$, and $d_{I,M}^{\mathrm{sym}}, d_{I,N}^{\mathrm{sym}}<\infty$.
	Let us define the topology on $N$ by $d_{I,N}^{\mathrm{sym}}$.
	Then the map $\mathcal{L}:N\to \Lipi(M, d_{I,M}^{\mathrm{sym}})$ defined by 
	$$\mathcal{L}(z) = \ell_z(\cdot)$$
	is continuous and injective.
\end{thm}
\begin{proof}
	The proof of continuity goes similarly to the case of horofunctions defined by distances.
	Let $x,y\in N$ and $\ell_x,\ell_y$ denote corresponding horofunctions.
	Then for any $\cdot\in M$, we have
	\begin{align*}
		|\ell_x(\cdot)-\ell_y(\cdot)|&\leq |I(\cdot,x)-I(b,x)-I(\cdot,y)+I(b,y)|\\
		&\leq |I(\cdot,x)-I(\cdot,y)| + |I(b,y)-I(b,x)|\\
		&\leq 2d_{I,N}^{\mathrm{sym}}(x,y)
	\end{align*}
	Hence the map $\mathcal{L} : N\to \Lipi(M, d_{I,M}^{\mathrm{sym}})$ is continuous.
	
	Suppose that $\ell_x = \ell_y$ for some $x,y\in M$.
	Then for any $\cdot\in M$,
	\begin{align}
		\ell_x(\cdot) = \ell_y(\cdot)&\iff I(\cdot,x)-I(b,x) = I(\cdot,y)-I(b,y)\nonumber\\
		&\iff I(\cdot,x)-I(\cdot,y) = I(b,x)- I(b,y).\label{eq.inj}
	\end{align}
	As we have fixed $x,y,b\in M$, the quantity $C:=I(b,x)-I(b,y)$ is a constant, 
	which implies that $x=y$ since $I$ separates points of $M$.
	Therefore, we see that $\mathcal{L}$ is injective.
\end{proof}
By Proposition \ref{prop.Lip-compact} and Theorem \ref{thm.lipi}, 
we see that the closure $\overline{\mathcal{L}(N)}$ is compact.
In this paper, we call $\overline{\mathcal{L}(N)}$ a {\em horo-compactification} of $N$,
which might be an abuse of terminology as $\mathcal{L}$ is {\em not necessarily a homeomorphism onto its image}.

These horo-compactifications are natural spaces when we consider associated distances.
\begin{prop}\label{prop.attain-sup}
The supremum in the definition of $d_{I,M}$ is attained in some element in the horo-compactification $\overline{\mathcal{L}(N)}\subset\Lipi(M,d_{I,M})$.
In other words, for any $x,y\in M$,
there exists $h\in \overline{\mathcal{L}(N)}$, we have
$$
d_{I,M}(x,y) = h(x) - h(y).
$$
\end{prop}
\begin{proof}
By Proposition \ref{prop.Lip-compact}, we see that $\overline{\mathcal{L}(N)}$ is a compact metrizable space.
Therefore, any sequence in $\overline{\mathcal{L}(N)}$ has a convergent subsequence.
Since 
$$
d_{I,M}(x,y) = \sup_{n\in N} I(x,n)-I(y,n) = \sup_{h\in \mathcal{L}(N)} h(x)-h(y), 
$$
 we have the conclusion.
\end{proof}

\subsection{Group actions}\label{sec.action}
Suppose that a group $G$ is acting on both $M,N$, and we have
$$
I(gm,gn) = I(m,n)
$$
for any $g\in G$ and $m\in M, n\in N$.
Then we may consider the group action on the space of horofunctions by 
$$
g\cdot h(x) : = h(g^{-1}x) - h(g^{-1}b),
$$
where $h\in \Lipi(M,d)$ for some distance $d$.

Then the following theorem is straightforward from the definitions:
\begin{thm}\label{thm.grp-action}
If a group $G$ acts on both $M,N$, and $I:M\times N\to\R$ is invariant under the diagonal action of $G$.
Then the distances (if defined) $d_{I,M}, d_{I,N}$ are also invariant under the diagonal action of $G$.

Furthermore, the action of $G$ extends continuously to the horo-compactifications via $I$.
\end{thm}

When we have an isometric action, one may consider the translation length:
$$
\tau_{d}(g) : = \lim_{n\to\infty}\frac{1}{n}d(g^{n}x,x).
$$
Similarly, we may consider
\begin{align*}
\tau_{I,M}(g) : = \limsup_{n\to\infty}\frac{1}{n}I(g^{n}x,y)\\
\tau_{I,N}(g) : = \limsup_{n\to\infty}\frac{1}{n}I(x,g^{n}y)\\
\end{align*}
Note that convergence of the limit in $\tau_{d}(g)$ is due to the triangle inequality.
Later, we prove that the limits in $\tau_{I,M}(g), \tau_{I,N}(g)$ exist under certain conditions in Proposition \ref{prop.horo-trans}.
As the notation suggests, $\tau_{I,M}(g)$ and $\tau_{I,N}(g)$ are independent of the choice of $x\in M$ and $y\in N$.
\begin{prop}\label{prop.transI}
Suppose that $d_{I,M}^{\mathrm{sym}},d_{I,N}^{\mathrm{sym}}<\infty$.
Then we see that $\tau_{I,M}, \tau_{I,N}$ are independent of $x\in M$ and $y\in N$.
\end{prop}
\begin{proof}
Let us fix basepoints $b_{M}\in M$, $b_{N}\in N$. Then for any $x\in M$, $y\in N$ we have
\begin{align*}
& ~|I(g^{n}x,y)-I(g^{n}b_{M}, b_{N})| \\
= & ~|I(g^{n}x,y) - I(g^{n}b_{M},y) &&+ I(g^{n}b_{M},y) - I(g^{n}b_{M}, b_{N})|\\
\leq & ~d_{I,M}^{\mathrm{sym}}(g^{n}x,g^{n}b_{M}) &&+ d_{I,N}^{\mathrm{sym}}(y,b_{N})\\
\leq & ~d_{I,M}^{\mathrm{sym}}(x,b_{M}) &&+ d_{I,N}^{\mathrm{sym}}(y,b_{N})\\
\end{align*}
Since 
$$
\lim_{n\to\infty}\frac{1}{n}\left(d_{I,M}^{\mathrm{sym}}(x,b_{M}) + d_{I,N}^{\mathrm{sym}}(y,b_{N})\right)=0
$$
we have established the conclusion for \( \tau_{I,M}(g) \).  
A similar argument applies to \( \tau_{I,N}(g) \).
\end{proof}

\subsection{North-south dynamics and translation lengths}
In this subsection, we give a condition to have $\tau_{I,M} = \tau_{d_{I,M}}$.
The content of this subsection is a generalization of 
\cite[Theorem 6.21, Theorem 7.10]{Masai}, in which Theorem \ref{thm.pAvol} below is proved.
We first recall the definition of a verison of north-south dynamics.
\begin{defi}
Suppose that a group $G$ is acting on a compact set $X$ continuously.
An element $g\in G$ has a north-south dynamics if there exists two fixed points $x_{+}\neq x_{-}\in X$ of $g$ such that
\begin{itemize}
\item $g^{n}(x)\to x_{+}$ for any $x\neq x_{-}$, and
\item $g^{-n}(x)\to x_{-}$ for any $x\neq x_{+}$.
\end{itemize}
\end{defi}
We consider the action of groups on the horo-compactifications given in Theorem \ref{thm.grp-action}.
From now on, we suppose the assumption of Theorem \ref{thm.grp-action} and Proposition \ref{prop.transI}, namely we assume
\begin{itemize}
\item a group $G$ acts on both $M,N$, and $I:M\times N\to\R$ is invariant under the diagonal action of $G$.
\item $d_{I,M},d_{I,N}<\infty$.
\end{itemize}
The second condition should be compared with Remark \ref{rmk.tri-d}, and the fact that for the discussion of translation length of standard distances, the triangle inequality plays an important role.

\begin{prop}[c.f. {\cite[Theorem 6.21]{Masai}}]\label{prop.horo-trans}
Suppose that $g\in G$ has a north-south dynamics on the horo-compactification $\overline{\mathcal{L}(N)}\subset\Lipi(M,d_{I,M})$ with fixed points $h_{+},h_{-}\in\overline{\mathcal{L}(N)}$.
Then for any $h\neq h_{-}\in \overline{\mathcal{L}(N)}$, we have
$$
\lim_{n\to\infty}\frac{1}{n}h(g^{-n}b_{M}) = h_{+}(g^{-1}b_{M}) = \tau_{I,M}(g^{-1}).
$$
In particular, the limits in the definitions of $\tau_{I,M}(g), \tau_{I,N}(g)$ exist.
\end{prop}
\begin{proof}
Let us suppose that $h\in\overline{\mathcal{L}(N)}$.
Since $g^{i}\cdot h(g^{-1}b_{M}) = h(g^{-i-1}b_{M}) - h(g^{-i}b_{M})$ we have
\begin{equation}
h(g^{-n}b_{M}) = \sum_{i=0}^{n-1}g^{i}\cdot h(g^{-1}b_{M}).\label{eq.horo-cocycle}
\end{equation} 
By north-south dynamics, when $h\neq h_{-}$, we have that $g^{i}\cdot h\to h_{+}$.
Hence for any $\epsilon>0$, there exists $K\in\N$ such that for any $k\geq K$, we have
$$
|g^{k}h(g^{-1}b_{M}) - h_{+}(g^{-1}b_{M})|<\epsilon.
$$
Therefore, 
\begin{align*}
&\lim_{n\rightarrow\infty}\frac{1}{n}\left|h(g^{-n}b_{M}) - nh_+(g^{-1}b_{M})\right|\\
= &\lim_{n\rightarrow\infty}\frac{1}{n}\left|\sum_{i=0}^{K}\left(g^i\cdot h(g^{-1}b_{M}) - h_+(g^{-1}b_{M})\right) + 
\sum_{i=K}^{n-1}\left(g^i\cdot h(g^{-1}b_{M})-h_+(g^{-1}b_{M})\right)\right|
 \leq \epsilon.
\end{align*}
As $\epsilon>0$ is arbitrary, we have 
\begin{equation}
\lim_{n\to\infty}\frac{1}{n}h(g^{-n}b_{M}) = h_{+}(g^{-1}b_{M}).\label{eq.nstrans}
\end{equation}
Note that we have \eqref{eq.nstrans} regardless $h\in\mathcal{L}(N)$ or $h\in\overline{\mathcal{L}(N)}\setminus\mathcal{L}(N)$.

Then suppose that $h\in\mathcal{L}(N)$, in other words $h(\cdot) = I(\cdot,y)-I(b_{M},y)$ for some $y\in N$.
Then by Proposition \ref{prop.transI}, we see that
\begin{equation}
\lim_{n\to\infty}\frac{1}{n}h(g^{-n}b_{M}) = \tau_{I,M}(g^{-1}).\label{eq.hhtrans}
\end{equation}
where the existence of the limit follows from \eqref{eq.nstrans} and Proposition \ref{prop.transI}.

Thus we have
$$
\lim_{n\to\infty}\frac{1}{n}h(g^{-n}b_{M}) = h_{+}(g^{-1}b_{M}) = \tau_{I,M}(g^{-1}).
$$
\end{proof}

\begin{coro}\label{coro.horo-trans}
With the same assumption as Proposition \ref{prop.horo-trans}, we have
$$
\lim_{n\to\infty}\frac{1}{n}h_{-}(g^{-n}b_{M}) = h_{-}(g^{-1}b_{M}) = -\tau_{I,M}(g).
$$
\end{coro}
\begin{proof}
Applying Proposition \ref{prop.horo-trans} to $g^{-1}$, and noting
$$
h_{-}(g^{-n}b_{M}) = g^{-n}\cdot h_{-}(g^{-n}b_{M}) = h_{-}(b_{M}) - h_{-}(g^{n}b_{M}) = - h_{-}(g^{n}b_{M})
$$
the claim follows.
\end{proof}

\begin{thm}[c.f. {\cite[Theorem 7.10]{Masai}}]\label{thm.ns-trans}
Suppose that 
\begin{itemize}
\item $g\in G$ has a north-south dynamics on the horo-compactification $\overline{\mathcal{L}(N)}\subset\Lipi(M,d_{I,M}^{\mathrm{sym}})$ with fixed points $h_{+},h_{-}\in\overline{\mathcal{L}(N)}$.
\item $\tau_{I,M},\tau_{I,N},\tau_{d_{I,M}}$ are all non-negative.
\end{itemize}
Then we have
\begin{enumerate}
\item if $\tau_{d_{I,M}}(g) > 0$, then
$$
\tau_{I,M}(g) =\tau_{I,M}(g^{-1}) = \tau_{d_{I,M}}(g).
$$
\item if $\tau_{d_{I,M}}(g)= 0$, then at least one of 
$\tau_{I,M}(g)$, or $\tau_{I,M}(g^{-1})$ is zero.
\end{enumerate}

\end{thm}
\begin{proof}
By Proposition \ref{prop.horo-trans} and the definition of $d_{I,M}$, we see that
\begin{equation}
\tau_{I,M}(g) \leq \tau_{d_{I,M}}(g).
\end{equation}
By Proposition \ref{prop.attain-sup}, we see that for any $k\in\N$ there exists $h_{k}\in\overline{\mathcal{L}(N)}$ such that
$$
d_{I,M}(b_{M},g^{k}b_{M}) = h_{k}(b_{M}) - h_{k}(g^{k}b_{M}) = - h_{k}(g^{k}b_{M})
$$
We define $F:\overline{\mathcal{L}(N)}\to\R$ by $F(h) = -h(g^{-1}b_{M})$.
By \eqref{eq.horo-cocycle}, we have
\begin{equation}
-h(g^{-k}b_{M}) = \sum_{i=0}^{k-1}(-g^{i}h(g^{-1}b_{M})) = \sum_{i=0}^{k-1} F(g^{i}\cdot h).\label{eq.F-cocycle}
\end{equation}
Now we define a Borel probability measure $\mu_{k}$ on $\overline{\mathcal{L}(N)}$ by
$$\mu_k=\frac{1}{k}\sum_{i=0}^{k-1}(g^i)_*\delta_{h_k}$$ where $\delta_{h}$ is the Dirac measure at $h\in\overline{\mathcal{L}(N)}$.
By \eqref{eq.F-cocycle}, we see that the $\mu_{k}$ satisfies
$$\int F(h)d\mu_k(h) = \frac{1}{k}d_{I,M}(b_{M},g^{-k}b_{M}).$$

Notice that by the triangle inequality, we have
	$$\frac{1}{k}d_{I,M}(b_{M},g^{-k}b_{M})\geq \lim_{i\rightarrow\infty}\frac{1}{i}d_{I,M}(b_{M},g^{-i}b_{M})=\tau_{d_{I,M}}(g).$$

As $\overline{\mathcal{L}(N)}$ is compact, by taking a subsequence if necessary, we have a weak limit point $\mu_{\infty}$ of the sequence $\{\mu_{k}\}$.
By the definition of the weak limit, we have
	$$\int F(h)d\mu_\infty(h)\geq \tau_{d_{I,M}}(g).$$

Furthermore, the definition of $\mu_k$ implies that $\mu_\infty$ is $g$ invariant, 
i.e. $$g_*\mu_\infty = \mu_\infty.$$
We now consider the space $\mathcal{M}_g(\overline{\mathcal{L}(N)})$ of $g$ invariant measures on $\overline{\mathcal{L}(N)}$.
As $\mathcal{M}_g(\overline{\mathcal{L}(N)})$ is convex and $\mu_\infty\in\mathcal{M}_g(\overline{\mathcal{L}(N)})$, the Krein-Milman theorem shows that there is an extreme point $\mu\in\mathcal{M}_g(\overline{\mathcal{L}(N)})$ such that
	$$\int F(h)d\mu(h)\geq \tau_{d_{I,M}}(g).$$
	The standard theory of ergodic measures says that ergodic measures in $\mathcal{M}_g(\overline{\mathcal{L}(N)})$ are precisely the extreme points.
	Therefore, we see that $\mu$ is ergodic.
	Then by the Birkhoff ergodic theorem, for $\mu$-a.e. $h\in\overline{\mathcal{L}(N)}$ we have
	$$\lim_{k\rightarrow\infty}\frac{1}{k}\sum_{i=0}^{k-1}F(h) = \int F(h)d\mu(h)\geq \tau_{d_{I,M}}(g).$$
	Hence, by (\ref{eq.F-cocycle}), 
	$$\lim_{k\rightarrow\infty}\frac{1}{k}(-h(g^{-k}b_{M}))\geq \tau_{d_{I,M}}(g).$$
	On the other hand, as $-h(g^{-k}b_{M})\leq d_{I,M}(b_{M},g^{-k}b_{M})$, we have 
	$$\lim_{k\rightarrow\infty}\frac{1}{k}(-h(g^{-k}b_{M}))\leq \tau_{d_{I,M}}(g).$$
	Hence, we have
	\begin{equation}
	\lim_{k\rightarrow\infty}\frac{1}{k}(-h(g^{-k}b_{M})) = \tau_{d_{I,M}}(g).\label{eq.t-len1}
	\end{equation}
	A similar argument applied for $F'(h):=h(gb_{M})$, we also have
	\begin{equation}
	\lim_{k\rightarrow\infty}\frac{1}{k}\cdot h(g^{k}b_{M}) = \tau_{d_{I,M}}(g).\label{eq.t-len2}
	\end{equation}
	Now, we compare \eqref{eq.t-len1} and \eqref{eq.t-len2} with Proposition \ref{prop.horo-trans} and Corollary \ref{coro.horo-trans}.
	Since $\tau_{d_{I,M}}(g)\geq 0$, we see that if $\tau_{d_{I,M}}(g)$ is positive then \eqref{eq.t-len1} can only hold when $h=h_{-}$. Hence we have
	$$	
	\tau_{I,M}(g) = \tau_{I,M}(g^{-1}) = \tau_{d_{I,M}}(g).
	$$
	
	Similarly if $\tau_{d_{I,M}}(g)$ is zero we must have
	at least one of $\tau_{I,M}(g^{-1})$,  $\tau_{d_{I,M}}(g)$ is zero.

\end{proof}

\subsection{Summary of results}

\begin{thm}\label{thm.summary}
Suppose $I:M\times N\to \R$ separates points of $M,N$ and 
$$d_{I,M}^{\mathrm{sym}},d_{I,N}^{\mathrm{sym}}<\infty.$$
Then we have the compactifications of $M$ and $N$ via the following maps.
In other words, the following maps are continuous and injective:
\begin{enumerate}
\item $N \hookrightarrow \Lipi(M,d_{I,M}^{\mathrm{sym}})$ defined by horofunctions by $I$ (Theorem \ref{thm.lipi}),
\item $M \hookrightarrow \Lipi(N,d_{I,N}^{\mathrm{sym}})$ defined by horofunctions by $I$ (Theorem \ref{thm.lipi}),
\item $N \hookrightarrow \Lipi(N,d_{I,N})$ defined by horofunctions by $d_{I,N}$, and
\item $M \hookrightarrow \Lipi(M,d_{I,M})$ defined by horofunctions by $d_{I,M}$.
\end{enumerate}

In particular, if $M=N$ then we have 
\begin{itemize}
\item
two a priori different distances on $M$, and
\item 
four a priori different horo-compactifications of $M$
\end{itemize}
via the map $I$.

If $G$ acts on both $M,N$ and preserves $I$, then distances are also preserved by the action of $G$, and the action of $G$ extends continuously to all the horo-compactifications given above.
Furthermore, if an element $g\in G$ has a north-south dynamics on the compactification of $N$ given as item (2) above, then the quantity
$$
\tau_{I,M}(g) : = \lim_{n\to\infty}\frac{1}{n}I(g^{n}x,y)\\
$$
exists.
If $\tau_{I,M}(g),\tau_{I,M}(g^{-1})$ are all nonnegative and the translation length 
$\tau_{d_{I,M}}(g)$ of $g$ with respect to the distance $d_{I,M}$ 
is positive, then $\tau_{d_{I,M}}(g)=\tau_{I,M}(g)=\tau_{I,M}(g^{-1})$.
\end{thm}

\section{Teichm\"uller space and distances}\label{sec.Tdists}
\subsection{Known distances}
From now on, we consider an orientable closed surface $S$ of genus $\geq 2$.
The Teichm\"uller space $\T$ of $S$ is the space of marked hyperbolic (or Riemann) surfaces. The mapping class group of $S$ is denoted as $\mcg$.
Let $\scc$ denote the isotopy classes of essential simple closed curves on $S$.
We first recall some known results.
\begin{prop}\label{prop.Th-Tei}
We consider the map $I:\T\times \scc\to \R$ by $I(X,\alpha) = \ell_{X}(\alpha)$.
\begin{enumerate}
\item When $\ell_{X}(\alpha)$ is the {\em hyperbolic length} of the curve $\alpha$, then $d_{I,\T}(X,Y)$ coincides with the Thurston Lipschitz distance \cite{Thurston}.
\item When $\ell_{X}(\alpha)$ is the {\em extremal length} of the curve $\alpha$, then $d_{I,\T}(X,Y)$ coincides with the Teichm\"uller distance (Kerckhoff's formula \cite{Kerckhoff}).
\end{enumerate}
\end{prop}
Proposition \ref{prop.Th-Tei} serves as our motivation for introducing Theorem \ref{thm.summary}.
However, if we use $\scc$, we may not consider $d_{I,\scc}$ as they are infinity.
In the next section, we consider the space of geodesic currents to have functions from $\TT$ and Proposition \ref{prop.Th-Tei} should be compared with Proposition \ref{prop.hyp} and Proposition \ref{prop.E1} below.

\subsection{Volume Functions}\label{sec.volumes}
In the rest of this section, let us consider volume functions on the space of quasi-Fuchsian manifolds. 
A Kleinian group $\Gamma<\PSL, \Gamma\cong\pi_{1}(S)$ is said to be {\em quasi-Fuchsian} if its limit set is a Jordan curve. 
The Jordan curve splits the Riemann sphere $\hat{\mathbb{C}}$ at infinity of the hyperbolic 3-space $\HH^{3}$ into two simply connected regions.
Each component of the complement determines a complex structure, and hence a point in $\T$, after a suitable adjustment of orientation.
Thus, we obtain a point on $\TT$ from a quasi-Fuchsian group.
A quasi-Fuchsian manifold is the quotient $\HH^{3}/\Gamma$ of a quasi-Fuchsian group.

By the Bers simultaneous uniformization theorem \cite{Bers-simul}, 
it turns out that the space of quasi-Fuchsian manifolds is parametrized by $\TT$.

Quasi-Fuchsian manifolds themselves are of infinite volume.
However, there are two natural notions of volumes of quasi-Fuchsian manifolds.
\begin{defi}[Volumes of quasi-Fuchsian manifolds]\label{defi.vol-dist}
There are two functions from $\TT$ defined via certain volumes of quasi-Fuchsian manifolds:
\begin{enumerate}
\item Convex core volume: quasi-Fuchsian manifolds are known to be convex co-compact, and hence the convex hull of the limit sets (a.k.a convex core) is compact. The volume of the convex core is finite and hence defines a symmetric map 
$$
V_{C}:\TT\to\R_{\geq 0}.
$$
Let $d_{C}:\TT\to\R_{\geq 0}$ denote the distance obtained by Theorem \ref{thm.summary}
\item Renormalized volume: the definition of renormalized volume is a bit involved, we refer to \cite{KS} for the details. The renormalized volume of quasi-Fuchsian manifolds is known to be finite and non-negative (c.f. \cite{Masai} and references therein), and hence defines a symmetric map
$$
V_{R}:\TT\to\R_{\geq 0}.
$$
Let $d_{R}:\TT\to\R_{\geq 0}$ denote the distance obtained by Theorem \ref{thm.summary}
\end{enumerate}
\end{defi}

Given a pseudo-Anosov mapping class $\varphi\in\mcg$, 
the mapping torus 
$$
M(\varphi):=S\times [0,1]/((\varphi(x),0)\sim (x,1))
$$
becomes a complete hyperbolic $3$-manifold by the work of Thurston.

As the mapping class group $\mcg$ acts on $\T$, and associated diagonal action preserves $V_{C}, V_{R}$.
The action of $\mcg$ on $\TT$ also preserves $d_{C}$ and $d_{R}$ (see \S \ref{sec.action}).

If we have an isometric action, we may consider the translation length:
$$
\tau(\varphi) :=\lim_{n\to\infty}\frac{1}{n}d(\varphi^{n}x,x).
$$
Let $\tau_{C}(\varphi)$ and $\tau_{R}(\varphi)$ denote 
the translation length of $\varphi\in\mcg$ for $d_{C}$ and $d_{R}$ respectively.

In \cite{Masai}, it is proved that:
\begin{thm}\label{thm.pAvol}
Let $\varphi\in \mcg$ be pseudo-Anosov.
Then the translation length $\tau_{R}(\varphi)$ with respect to the distance $d_{R}$ coincides with the hyperbolic volume of the mapping torus $M(\varphi)$, in short
$$
\tau_{R}(\varphi) = \Vol(M(\varphi)).
$$
\end{thm}
In \cite{Masai}, the author studies a variant of horo-compactification via the renormalized volume $V_{R}$, and the horo-compactification so obtained was the key ingredient of the proof of Theorem \ref{thm.pAvol}.

We now prove a similar result for the convex core volume version.
\begin{thm}\label{thm.volcc}
There exists a constant $A$ which depends only on the topology of $S$ such that
\[
|d_{R}(X,Y)-d_{C}(X,Y)|<A
\]
for any $X,Y\in\T$. In particular, for any pseudo-Anosov $\varphi\in \mcg$, we have
\[
\tau_{C}(\varphi) = \Vol(M(\varphi)).
\]
\end{thm}
\begin{proof}
By the result of Schlenker \cite{Schlenker}, 
it is observed that there is a constant $A'$ such that $|V_{R}(X,Y)-V_{C}(X,Y)|<A'$ for any $X,Y\in\T$.
By definition of $d_{R}$ and $d_{C}$, we have $|d_{R}(X,Y)-d_{C}(X,Y)|<2A'$.
This implies that $\tau_{R}(\varphi) = \tau_{C}(\varphi)$.
Hence by Theorem \ref{thm.pAvol}, we have
$$
\tau_{C}(\varphi) = \Vol(M(\varphi)).
$$
\end{proof}
\section{Geodesic currents and length functions}\label{sec.curr}
\subsection{Intersection number and hyperbolic length}
In this section, we recall the geodesic currents introduced by Bonahon \cite{Bonahon}.
Let $S$ be an orientable closed surface $S$ of genus $\geq 2$.
We equip $S$ with a reference hyperbolic structure.
Then $S$ is realized as a quotient $\HH/\Gamma$ where $\Gamma\cong\pi_{1}(S)$ is a Fuchsian group.
We define the space of (unoriented) geodesics by 
$$
\mathcal{G}:=(\partial\HH\times\partial\HH)/(\Z/2\Z)
$$
where the action of $(\Z/2\Z)$ switches the coordinates.
\begin{defi}[Geodesic currents]
A {\em geodesic current} is a positive, locally finite, $\Gamma$-invariant Radon measure on $\mathcal{G}$.
We denote by $\curr$ the space of geodesic currents equipped with weak$^{\ast}$ topology.
The quotient via multiplicative action of $\R_{>0}$ is denoted $\mathbb{P}\curr$.
\end{defi}
Note that $\curr$ is known to be independent of the reference hyperbolic structure.

\begin{ex}
We give some examples of geodesic currents.
\begin{enumerate}
\item Given a closed curve  $\gamma$  on  S , the union of the endpoints  $\partial \widetilde{\gamma}$  of all of its lifts  $\widetilde{\gamma} \subset \mathbb{H} $ forms a locally finite, $\Gamma$-invariant subset of  $\mathcal{G}$. Consequently, the sum of Dirac measures on $\partial\widetilde\gamma$  provides an example of a geodesic current.
By abuse of notations, we denote the geodesic current by the same symbol $\gamma$.
\item For each hyperbolic structure \( X \) on \( S \), the volume form on the unit tangent bundle induces a geodesic current, known as the \emph{Liouville current}, which we denote by \( L_X \), see \cite{Bonahon} for more detail.
\end{enumerate}
\end{ex}

\bigskip
{\bf Notation.}
Let \( X \in \mathcal{T} \). By a slight abuse of notation, we denote by \( X \in \curr \) the corresponding Liouville current \( L_{X} \) whenever no confusion arises.

\bigskip

Recall that given two closed curves \( \gamma, \delta \) on the surface \( S \), we have the geometric intersection number \( i(\gamma, \delta) \).  
Let \( \mathcal{T} \) denote the Teichm\"uller space of \( S \), which is the space of marked hyperbolic structures.
\begin{thm}[Bonahon \cite{Bonahon}]\label{thm.Bonahon}
The following statements hold.
\begin{enumerate}
    \item The set of weighted closed curves is dense in \( \curr \).
    \item The geometric intersection number of closed curves extends continuously to \( \curr \).
    \item The map \( T \ni X \mapsto L_{X} \in \mathbb{P}\curr \) is an embedding and its closure coincides with the Thurston compactification.
    \item For a closed curve \( \gamma \) and a Liouville current \( L_{X} \), we have  
          \[
          i(\gamma, L_{X}) = \ell_{X}(\gamma),
          \]  
          where \( \ell_{X}(\gamma) \) is the hyperbolic length of \( \gamma \).
   \item For any $X$, we have $i(L_{X},L_{X}) = \pi^{2}|\chi(S)|$ where $\chi(S)$ is the Euler characteristic of $S$.
\end{enumerate}
\end{thm}

\subsection{Extremal length}
Let $X\in\T$ be a Riemann surface and \(\gamma \) be a closed curve.
The {\em extremal length} of \(\gamma \) is defined as

\[
\ext_{X}(\gamma) := \sup_{\rho} \frac{\ell_{\rho}(\gamma)^{2}}{\operatorname{Area}(\rho)}
\]

where
\begin{itemize}
\item the supremum is taken over all conformal metrics \( \rho(z) \) on $X$.
\item  $L_{\rho}(\gamma):=\int_{\gamma}\rho|dz|$ is the $\rho$-length of a path $\gamma$, and $\ell_{\rho}(\gamma) = \inf_{\gamma'}L_{\rho}(\gamma')$ where the infimum is taken over all $\gamma'$ homotopic to $\gamma$.

\item \( \mathrm{Area}(\rho) := \int_{X}\rho^{2}dxdy \) is the area with respect to $\rho$.
\end{itemize}
Any conformal metric that attains the extremal length is called an {\em extremal metric}.
By the work of Jenkins and Strebel (see e.g. \cite{GL}), 
it is known that for any simple closed curve $\alpha$, there exists a quadratic differential $q = q(X,\alpha)$ on $X$ whose associated singular Euclidean metric is the extremal metric of $\alpha$.

By the work of Mart\'inez--Granado and Thurston, we have the extremal lengths of geodesic currents.
\begin{thm}[{\cite[Section 4.8]{GT}}]\label{thm.GT}
Let $X\in\T$.
Then the square root of the extremal length function $\sqrt{\ext_{X}}$ extends continuously to $\curr$, namely, we have a continuous function
$$
\sqrt{\ext_{X}}:\curr\to\R_{\geq 0}.
$$
\end{thm}

For later use, let us extend Minsky's inequality.
\begin{lem}[Minsky's inequality c.f. {\cite[Lemma 5.1]{Minsky}}]\label{lem.Minsky}
Let $\alpha$ be a measured foliation, and $G\in\curr$.
Then for any $Z\in\T$, we have
\begin{equation}
i(\alpha,G)\leq \sqrt{\ext_{Z}(\alpha)}\cdot\sqrt{\ext_{Z}(G)}\label{eq.Minsky}
\end{equation}
\end{lem}
\begin{proof}
The proof follows almost identically to Minsky's original proof of the inequality when \( \alpha \) and \( G \) are simple closed curves. Since weighted simple closed curves are dense in the space of measured foliations (see e.g. \cite[Proposition 6.18]{FLP}), we may assume that \( \alpha \) is a simple closed curve.  
Notice that both sides of \eqref{eq.Minsky} are homogeneous with respect to the weight on $\alpha$.  

There exists a holomorphic quadratic differential, called the Jenkins-Strebel differential \( q := q_{Z}(\alpha) \) on \( Z \), such that the associated singular Euclidean metric is the extremal metric for \( \alpha \) on \( Z \).  

The horizontal foliation of \( q \) is homotopic to \( \alpha \), and the horizontal and vertical foliations of \( q \) form a flat cylinder \( C \) of height \( 1 / \sqrt{\operatorname{Ext}_{Z}(\alpha)} \) and circumference \( \sqrt{\operatorname{Ext}_{Z}(\alpha)} \).  

Now, suppose that \( G \) is a closed curve. The length \( \ell_{q}(G) \) of \( G \) satisfies  
\[
\ell_{q}(G) \geq i(\alpha, G) \cdot 1/\sqrt{\operatorname{Ext}_{Z}(\alpha)}
\]
since \( G \) passes through the cylinder \( C \) exactly \( i(\alpha, G) \) times.  

Then, by the definition of extremal length, we obtain  
\[
\operatorname{Ext}_{Z}(G) \geq \ell_{q}(G)^{2} \geq \frac{i(\alpha, G)^{2}}{\operatorname{Ext}_{Z}(\alpha)}
\]
which implies
\[
i(\alpha, G) \leq \sqrt{\operatorname{Ext}_{Z}(\alpha)} \cdot \sqrt{\operatorname{Ext}_{Z}(G)}.
\]
Since weighted closed curves are dense in \( \curr \), we conclude that \eqref{eq.Minsky} holds for any \( G \in \curr \).
\end{proof}

We also note the following:
\begin{thm}[\cite{ExHyp}]\label{thm.ExHyp}
Let $X\in\T$ and $L_{X}$ denote the corresponding Liouville current.
Then the extremal length of $L_{X}$ is determined by the topology of $X$.
\[
\ext_{X}(L_X) =\frac{\pi}{2}~i(L_{X},L_{X}) =  \frac{~\pi^{2}}{4} \operatorname{Area}(X) = \frac{~\pi^{3}}{2}|\chi(S)|.
\]
\end{thm}

\section{Distances via length functions}\label{sec.dist-len}
Via Liouville currents, we have an embedding of $\T$ into $\curr$.
If a length function $\ell$ is determined by an element of $\T$ and extends continuously to $\curr$, then we have a map
$$
I:\TT\to\R, \quad
(X, Y)\mapsto \log(\ell_{X}(Y))
$$
By Theorem \ref{thm.summary}, we may have a priori different distances.
\subsection{Thurston's metric}
First, consider the hyperbolic length function $\ell_{X}$.
In this case, $I:\TT\to\R, (X, Y)\mapsto \ell_{X}(Y)$ is just the restriction of the intersection number $i(\cdot,\cdot)$ by Theorem \ref{thm.Bonahon}.
Since $i(\cdot,\cdot)$ is symmetric, Theorem \ref{thm.summary} applied to 
$\log\circ i:\TT\to\R$ yields a metric:
$$
d_{\mathrm{hyp}}(X,Y):=\sup_{Z\in\T}\log\frac{i(X,Z)}{i(Y,Z)}.
$$
\begin{prop}\label{prop.hyp}
The distance $d_{\mathrm{hyp}}$ coincides with Thurston's asymmetric distance $d_{L}$ and 
horo-compactification via 
$$\T\to\Lipi(\T,d_{\mathrm{hyp}}), \quad Z\mapsto \log\left( i(\cdot, Z)\right)-\log\left(i(b,Z)\right)$$
coincides with the Thurston compactification.
\end{prop}
\begin{proof}
Thurston's asymmetric Lipschitz distance $d_{L}$ is defined in \cite{Thurston} as the minimal Lipschitz constant among homeomorphisms compatible with markings.
Hence, for any weighted closed curve $a\subset S$, we have
$$
\log\frac{i(X,a)}{i(Y,a)}\leq d_{L}(X,Y)
$$
As weighted closed curves are dense, we see that $d_{\mathrm{hyp}}\leq d_{L}$.

By Thurston's characterization (Proposition \ref{prop.Th-Tei}(1)), we see that $d_{L}$ is realized by taking the supremum over the simple closed curves.
As we know (Theorem \ref{thm.Bonahon} (3)) that closure of $\T$ in $\mathbb{P}\curr$ is the Thurston compactification, which contains all the simple closed curves, we see that $d_{\mathrm{hyp}}=d_{L}$.

By the continuity of $i(\cdot,\cdot)$ and the fact that the map $\log\left(i(\cdot, Z)\right)-\log\left(i(b,Z)\right)$ is invariant under multiplication by positive numbers on $\curr$, 
$$\mathbb{P}\curr\supset\T\to\Lipi(\T,d_{\mathrm{hyp}}), \quad Z\mapsto \log\left(i(\cdot, Z)\right)-\log\left(i(b,Z)\right)$$
extends continuously to the Thurston compactification.
By Theorem\ref{thm.summary} and the fact that the hyperbolic length of simple closed curves characterizes the hyperbolic structure, we see that
$$
\overline{\T}\to\Lipi(\T,d_{\mathrm{hyp}})
$$
is a continuous injective map from a compact space to a Hausdorff space, which in turn is a homeomorphism onto its image.
\end{proof}
\begin{rmk}
Proposition \ref{prop.hyp} recovers Walsh's characterization of the Thurston boundary \cite{Walsh} in terms of intersection numbers.
\end{rmk}

For later discussion, we give a lower bound for the intersection number 
between Liouville currents in terms of Thurston's distance. 
This estimate may be of independent interest. 
The intersection number $i(X,Y)$ reflects the Weil--Petersson geometry 
through its infinitesimal behavior (see~\cite[Theorem~19]{Bonahon}). 
In contrast, Proposition~\ref{prop.int-est} provides a large-scale estimate: 
it relates $\log i(X,Y)$ to Thurston's Lipschitz distance $d_L(X,Y)$ 
by explicit inequalities valid throughout Teichm\"uller space. 
The appearance of the systole in the bounds reflects the behavior 
in the thin part.
\begin{prop}[Intersection number v.s. Lipschitz distance]\label{prop.int-est}
For any $X,Y\in\T\subset\curr$, we have
$$d_{L}(Y,X) - \log\left(\frac{4}{\mathrm{sys}(Y)^{2}}\right) 
\leq  \log i(X,Y) \leq
d_{L}(Y,X)+\log\left(\pi^{2}|\chi(S)|\right)
$$
\end{prop}
where $\mathrm{sys}(Y)$ denote the systole of $Y$, i.e. the hyperbolic length of the shortest closed geodesic on $Y$.
\begin{proof}
We use the following estimate due to Torkaman \cite[Proposition 2.4]{Torkaman}.
For any $G_{1}, G_{2}\in\curr$ and $Y\in\T$, we have
\begin{equation}
i(G_{1}, G_{2})\leq \frac{4}{\mathrm{sys}(Y)^{2}}i(Y,G_{1})i(Y,G_{2}).
\end{equation}
Setting $G_{2} =X$, we have
\begin{equation}
\frac{i(X, G_{1})}{i(Y, G_{1})}\leq \frac{4}{\mathrm{sys}(Y)^{2}}\cdot i(Y,X).\label{eq.Torkaman}
\end{equation}
Since \eqref{eq.Torkaman} holds for any $G_{1}\in\curr$, we have
$$d_{L}(Y,X) - \log\left(\frac{4}{\mathrm{sys}(Y)^{2}}\right) \leq  \log i(X,Y)$$
by Proposition \ref{prop.hyp}.

The upper bound is obtained by Proposition \ref{prop.hyp} and Theorem \ref{thm.Bonahon}:
\begin{align*}
&\log\frac{i(X,Y)}{i(Y,Y)}\leq \sup_{Z\in\T}\log\frac{i(X,Z)}{i(Y,Z)} = d_{L}(Y,X)\\
\Longrightarrow &\log(i(X,Y))-\log\left(\pi^{2}|\chi(S)|\right)\leq d_{L}(Y,X).
\end{align*}
\end{proof}

\subsection{Teichm\"uller metric and its horofunction counterpart}\label{sec.counterpart}
Thanks to the extension of the square root of the extremal length to $\curr$ (Theorem \ref{thm.GT}), we may define
$$
I_{E}:\TT\to\R, \quad (X,Y)\mapsto \frac{1}{2}\log\left(\ext_{X}(Y)\right)
$$
where $Y$ is identified with the Liouville current.

\begin{defi}[distances via extremal length]\label{defi.E12}
As $I_{E}$ is asymmetric, we may have two distances by Theorem \ref{thm.summary}.
\begin{itemize}
\item $\displaystyle d_{E,1}(X,Y):=\sup_{Z\in\T} I_{E}(X,Z) - I_{E}(Y,Z) = \frac{1}{2}\sup_{Z\in\T}\log\frac{\ext_{X}(Z)}{\ext_{Y}(Z)}$.
\item $\displaystyle d_{E,2}(X,Y):=\sup_{Z\in\T} I_{E}(Z,X) - I_{E}(Z,Y) = \frac{1}{2}\sup_{Z\in\T}\log\frac{\ext_{Z}(X)}{\ext_{Z}(Y)}$.
\end{itemize}
\end{defi}
We first prove that $d_{E,1},d_{E,2}$ are finite and are distances.
\begin{prop}\label{prop.E1}
The distance $d_{E,1}$ coincides with the Teichm\"uller distance.
\end{prop}
\begin{proof}
The proof goes similarly to Proposition \ref{prop.hyp}.
As the Teichm\"uller distance is defined via dilatation of the quasi-conformal mappings and the fact that quasi-conformal maps distort extremal length up to the multiplication of the dilatations, we see that for any weighted closed curve $a$,
$$
\frac{1}{2}\log\frac{\ext_{X}(a)}{\ext_{Y}(a)}\leq \dt(X,Y).
$$
Hence, the density of the weighted closed curves implies that $d_{E,1}\leq \dt$.
By Kerckhoff's formula (Proposition \ref{prop.Th-Tei}(2)), and the fact that $\sqrt{\ext_{X}}(\cdot)$ is extended continuously to $\curr$, we see that $d_{E,1} = \dt$.
\end{proof}

Since $d_{E,1}$ is the Teichm\"uller distance, we call $d_{E,2}$ a {\em horofunction counterpart} to the Teichm\"uller distance.

The proof that the counterpart $d_{E,2}$ is a distance is more subtle.
Following Miyachi \cite{ray2}, we define a function as follows.
We fix a basepoint $b\in\T$.
\begin{defi}
We define \(\mathcal{E}:\T\times\curr\to\R_{>0}\) by
\[
\mathcal{E}(Z,G):=\frac{\sqrt{\ext_{Z}(G)}}{\exp(\dt(b,Z))}
\]
\end{defi}
\begin{lem}[c.f. {\cite[Lemma 1, and its remark]{ray2}}]\label{lem.decrese}
Suppose that $\gamma:\R_{\geq 0}\to\T$ is a Teichm\"uller geodesic ray with $\gamma(0) = b$.
Then for any $G\in\curr$, the map
\[
[0,\infty]\to\R, ~~t\mapsto \mathcal{E}(\gamma(t),G)
\]
is positive and non-increasing.
\end{lem}
\begin{proof}

Let \( t \geq s \geq 0 \). Then by Proposition \ref{prop.E1}, we have  
\begin{align*}
    & \frac{\operatorname{Ext}_{\gamma(t)}(G)}{\operatorname{Ext}_{\gamma(s)}(G)}
    \leq \sup_{F \in \curr} \frac{\operatorname{Ext}_{\gamma(t)}(F)}{\operatorname{Ext}_{\gamma(s)}(F)}
    = \exp(2 d_{T}(\gamma(t), \gamma(s))) 
    = \frac{\exp(2d_{T}(\gamma(t), b))}{\exp(2d_{T}(\gamma(s), b))} \\
    \Longrightarrow & \frac{\sqrt{\operatorname{Ext}_{\gamma(t)}(G)}}{\exp(d_{T}(\gamma(t), b))}
    \leq \frac{\sqrt{\operatorname{Ext}_{\gamma(s)}(G)}}{\exp(d_{T}(\gamma(s), b))} 
    \iff  \mathcal{E}(\gamma(t), G) \leq \mathcal{E}(\gamma(s), G).
\end{align*}
\end{proof}

Recall that the Teichm\"uller geodesic rays based at $b\in\T$ are determined by the element in $\mf$.
\begin{defi}[c.f. {\cite[\S 3.2]{ray2}}]
Let $\alpha\in\mf$ and $\gamma_{\alpha}:\R_{\geq 0}\to\T$ denote the Teichm\"uller ray determined by $\alpha$. Then we define
\[
e:\mf\times\curr\to\R_{\geq 0},~~(\alpha, G)\mapsto \lim_{t\to\infty}\mathcal{E}(\gamma_{\alpha}(t),G).
\]
The limit exists due to Lemma \ref{lem.decrese}.
\end{defi}

\begin{prop}\label{prop.keyestimate}
Fix \(\alpha\in\mf\). 
Then we have
\[
\frac{i(\alpha, G)}{\sqrt{\ext_{b}(\alpha)}}\leq e({\alpha},G) \leq 
\mathcal{E}(\gamma(t),G).
\]
In particular, it follows
\[
\quad i(\alpha, G) > 0 \implies e({\alpha},G) > 0.
\]
\end{prop}
\begin{proof}
By Minsky's inequality (Lemma \ref{lem.Minsky}) 

\begin{align*}
i(\alpha, G)^2 \leq &  ~\ext_{\gamma_{\alpha}(t)}(\alpha) \cdot \ext_{\gamma_{\alpha}(t)}(G)\\
=  & ~\frac{\ext_{\gamma_{\alpha}(t)}(G)}{e^{2t}}\cdot e^{2t}\cdot\ext_{\gamma_{\alpha}(t)}(\alpha)\\
=  & ~\mathcal{E}(\gamma_{\alpha}(t), G)^{2}\cdot\ext_{b}(\alpha)\\
\end{align*}
which implies
\[
\frac{i(\alpha, G)^2}{\ext_{b}(\alpha)}  \leq \mathcal{E}(\gamma_{\alpha}(t), G)^{2}
\]
Since the left hand side is independent of $t$, and \(\mathcal{E}(\gamma_{\alpha}(t), G)\) is non-increasing (Lemma \ref{lem.decrese}), we have
\[\frac{i(\alpha, G)}{\sqrt{\ext_{b}(\alpha)}}\leq e({\alpha},G) \leq 
\mathcal{E}(\gamma(t),G).\]

That \( i(\alpha, G) > 0 \) implies \( e({\alpha},G) > 0 \) is immediate.
\end{proof}

We now recall the work of Matsuzaki. Let $Z\in\T$
\begin{thm}[\cite{Matsuzaki}]\label{thm.Matsuzaki}
Let us define $\nu(Z)$ by
\[
\nu(Z) := \sup \left\{ \frac{\ext_{Z}(\alpha)}{i(Z,\alpha)^2} \ \middle| \ [\alpha] \in \mathcal{S} \right\},
\]
and let $\mathrm{sys}(Z)$ denote the hyperbolic systole of $Z$.

Then there exist universal constants \( r_0 \) and \( r_1 \) such that for any $Z\in\T$,
\[
\frac{1}{\pi \cdot\mathrm{sys}(Z)} \leq \nu(Z) \leq \max \left\{ \frac{r_0}{\mathrm{sys}(Z)}, r_1 \right\}.
\]
\end{thm}

\begin{thm}\label{thm.E2bound}
For any \( G_1, G_2 \in\curr \), \( \alpha\in\mf \), and $Z\in\T$, we have:
\[
\frac{\sqrt{\ext_{Z}(G_{1})}}{\sqrt{\ext_{Z}(G_{2})}} 
\leq 
\frac{i(b, \alpha)}{i(G_2, \alpha)} \cdot \sqrt{\nu(b)} \cdot\sqrt{\ext_b(G_1)} 
\]
In particular, if $G_{2}\in\T\subset\curr$, then 
\[
\frac{\sqrt{\ext_{Z}(G_{1})}}{\sqrt{\ext_{Z}(G_{2})}} 
\leq \exp(d_{L}(G_{2},b))\cdot \sqrt{\nu(b)} \cdot\sqrt{\ext_b(G_1)}.
\]
The right hand side $\exp(d_{L}(G_{2},b))\cdot \sqrt{\nu(b)} \cdot\sqrt{\ext_b(G_1)}$ is independent of $Z$.
\end{thm}
\begin{proof}
We consider the Teichm\"uller geodesic ray $\gamma$ based at $b$ that passes through $Z\in\T$.
Let $\alpha\in\mf$ denote the endpoint of $\gamma$, namely $\gamma=\gamma_{\alpha}$.
Note that $Z = \gamma(t)$ for some $t$.
Then by Proposition \ref{prop.keyestimate}, we have
\[
\frac{i(G, \alpha)}{\sqrt{\ext_{b}(\alpha)}}\leq \mathcal{E}(Z,G)\iff
\frac{1}{\mathcal{E}(Z,G)}\leq \frac{\sqrt{\ext_{b}(\alpha)}}{i(G, \alpha)}
\]
and $\mathcal{E}(Z,G)\leq \mathcal{E}(b,G) = \sqrt{\ext_{b}(G)}$.

Hence we have
\begin{align*}
\frac{\sqrt{\ext_{Z}(G_{1})}}{\sqrt{\ext_{Z}(G_{2})}} 
=\frac{\mathcal{E}(Z,G_{1})}{\mathcal{E}(Z,G_{2})}
\leq &\frac{i(b, \alpha)}{i(G_{2}, \alpha)}
\frac{\sqrt{\ext_{b}(\alpha)}}{i(b, \alpha)}
\sqrt{\ext_{b}(G_{1})}\\
\leq &\frac{i(b, \alpha)}{i(G_{2}, \alpha)}
\sqrt{\nu(b)}
\sqrt{\ext_{b}(G_{1})}.
\end{align*}
If $G_{2}\in\T$, then Proposition \ref{prop.hyp} implies
\[
\frac{i(b, \alpha)}{i(G_{2}, \alpha)}\leq \exp(d_{L}(G_{2},b)).
\]
Thus, we have
\[
\frac{\sqrt{\ext_{Z}(G_{1})}}{\sqrt{\ext_{Z}(G_{2})}} 
\leq \exp(d_{L}(G_{2},b))\cdot \sqrt{\nu(b)} \cdot\sqrt{\ext_b(G_1)}.
\]
By Theorem \ref{thm.Matsuzaki}, we see that the upper bound \(\exp(d_{L}(G_{2},b))\cdot \sqrt{\nu(b)} \cdot\sqrt{\ext_b(G_1)}.\) is independent of $Z$.
\end{proof}

By Theorem \ref{thm.E2bound}, we see that $d_{E,2}$ is finite.
We next prove the point separation property.
\begin{lem}\label{lem.pA}
Let $\varphi\in\mcg$ be a pseudo-Anosov mapping class and let 
$$\mathcal{F}_{-}(\varphi), \mathcal{F}_{+}(\varphi)\in\mf$$ denote its
stable and unstable foliations so that 
\begin{enumerate}
\item 
$
\displaystyle\lim_{n\to\infty}\varphi^{\pm n}(\alpha) = \mathcal{F}_{\pm}(\varphi)
$ in $\mathbb{P}\curr$ for any $\alpha\neq \mathcal{F}_{\mp}(\varphi)\in\mathbb{P}\curr$.
\item $\displaystyle\varphi\mathcal{F}_{\pm}(\varphi) = \lambda^{\pm 1}\mathcal{F}_{\pm}(\varphi)$ in $\mf\subset\curr$.
\item $\displaystyle i(\mathcal{F}_{-}(\varphi), \mathcal{F}_{+}(\varphi)) = 1$.
\end{enumerate}

Then for any current $G\in\curr\setminus{\mathcal{F}_{+}(\varphi)}$, we have
$$\lim_{n\to\infty}\frac{\varphi^{-n}G}{\lambda^{n}} = i(G, F_{+}(\varphi))\cdot \mathcal{F}_{-}(\varphi)\in\mf.$$
\end{lem}
\begin{proof}
As the intersection number is invariant under the diagonal action of $\mcg$, we have
\begin{align}
i\left(\frac{\varphi^{-n}G}{\lambda^{n}},\mathcal{F}_{+}(\varphi)\right) = 
i\left(G,\frac{\varphi^{n}\mathcal{F}_{+}(\varphi)}{\lambda^{n}}\right) = i(G,\mathcal{F}_{+}(\varphi)).\label{eq.pA-foli}
\end{align}
Note that by the north-south dynamics of pseudo-Anosovs on $\mathbb{P}\curr$ (see e.g. \cite[Theorem 11.5.5]{EU}), there exists $C=C(G,\varphi)>0$ such that
$$\lim_{n\to\infty}\frac{\varphi^{-n}G}{\lambda^{n}} = C\cdot \mathcal{F}_{-}(\varphi)\in\mf.$$
Then by \eqref{eq.pA-foli} and $ i(\mathcal{F}_{-}(\varphi), \mathcal{F}_{+}(\varphi)) = 1$, we see that 
$$
C = i(C\cdot F_{-}(\varphi), F_{+}(\varphi)) = \lim_{n\to\infty} i\left(\frac{\varphi^{-n}G}{\lambda^{n}}, F_{+}(\varphi)\right) = i(G, F_{+}(\varphi)).
$$
Therefore we have 
$$\lim_{n\to\infty}\frac{\varphi^{-n}G}{\lambda^{n}} = i(G, F_{+}(\varphi))\cdot \mathcal{F}_{-}(\varphi)\in\mf.$$
\end{proof}

\begin{coro}\label{coro.sp}
Suppose that we are given $X\neq Y\in\T$, then there exists $Z\in\T$ such that 
\[
\log\frac{\ext_{Z}(L_{X})}{\ext_{Z}(L_{Y})} > 0.
\]
\end{coro}
\begin{proof}
Given two distinct points $X,Y\in\T$, we have a pseudo-Anosov mapping class $\varphi$ so that $i(L_{X},\mathcal{F}_{+}(\varphi))> i(L_{Y}, \mathcal{F}_{+}(\varphi))$
by the fact that the set of stable and unstable foliations of pseudo-Anosov mapping classes is dense in $\pmf$ (see e.g. \cite[Theorem 6.19]{FLP}).
Let us choose the base point $b\in\T$ on the Teichm\"uller geodesic axis of $\varphi$.
Then, by Lemma \ref{lem.pA},
\begin{align*}
\mathcal{E}_{\varphi^{n}b}(L_{X}) &= 
\frac{\sqrt{\ext_{\varphi^{n}b}(L_{X})}}{\lambda^{n}} = \sqrt{\ext_{b}(\varphi^{-n}L_{X}/\lambda^{n})}
= \mathcal{E}_{b}\left(\frac{\varphi^{-n}L_{X}}{\lambda^{n}}\right)\\\nonumber
&\xrightarrow{n \to \infty}
i(L_{X}, F_{+}(\varphi))\cdot\mathcal{E}_{b}(\mathcal{F}_{-}(\varphi)).
\end{align*}
As we assumed $i(L_{X},\mathcal{F}_{+}(\varphi))> i(L_{Y}, \mathcal{F}_{+}(\varphi))$, 
\begin{align}
\nonumber \frac{1}{2}\log\frac{\ext_{\varphi^{n}b}(X)}{\ext_{\varphi^{n}b}(Y)}= &
\log\frac{\mathcal{E}_{\varphi^{n}b}(X)}{\mathcal{E}_{\varphi^{n}b}(Y)} \\\label{eq.pA-E2}
\xrightarrow{n \to \infty} &
\log\frac{i(L_{X}, F_{+}(\varphi))\cdot\mathcal{E}_{b}(\mathcal{F}_{-}(\varphi))}{i(L_{Y}, F_{+}(\varphi))\cdot\mathcal{E}_{b}(\mathcal{F}_{-}(\varphi))} 
= &
\log\frac{i(L_{X}, F_{+}(\varphi))}{i(L_{Y}, F_{+}(\varphi))}>0.
\end{align}
Hence for large enough $n$, we have 
$$\log\frac{\ext_{\varphi^{n}b}(L_{X})}{\ext_{\varphi^{n}b}(L_{Y})}>0.$$
\end{proof}

\begin{thm}[Thurston's Lipschitz distance v.s. $d_{E,2}$]\label{thm.E2}
The function $d_{E,2}:\TT\to\R_{\geq 0}$ is a (possibly asymmetric) distance and we have
\begin{equation}\label{eq.E2est}
d_{L}(Y,X)\leq d_{E,2}(X,Y)\leq d_{L}(Y,b) + 
\frac{1}{2}\left(\log\nu(b) + \log{\ext_b(L_{X})}\right).
\end{equation}
for any $b\in\T$.
In particular, we have
\begin{equation}\label{eq.E2est2}
d_{L}(Y,X)\leq d_{E,2}(X,Y)\leq d_{L}(Y,X) + 
\frac{1}{2}\left(\log\nu(X) + \log\left(\frac{~\pi^{3}}{2}|\chi(S)|\right)\right).
\end{equation}

Furthermore, 
let $\varphi\in\mcg$ be a pseudo-Anosov mapping class with dilatation $\lambda$.
Then we have
$$
\tau(\varphi,d_{E,2}):=\lim_{n\to\infty}\frac{1}{n}d_{E,2}(X,\varphi^{n} X) = \log\lambda.
$$
\end{thm}
\begin{proof}
Theorem \ref{thm.E2bound} and Corollary \ref{coro.sp} imply that 
$d_{E,2}<\infty$ and $\ext_{\cdot}(\cdot)$ separates points of $\T$.
Then Theorem \ref{thm.summary} shows that $d_{E,2}$ is a (possibly asymmetric) distance.
The upper bound in \eqref{eq.E2est} follows by 
Theorem \ref{thm.E2bound}.
The estimate \eqref{eq.E2est2} follows by letting $b = X$ and referring to Theorem \ref{thm.ExHyp}.

For the lower bound, we see in \eqref{eq.pA-E2} that
$$
\frac{1}{2}\log\frac{\ext_{\varphi^{n}b}(L_{X})}{\ext_{\varphi^{n}b}(L_{Y})}
\to 
\log\frac{i(L_{X}, F_{+}(\varphi))}{i(L_{Y}, F_{+}(\varphi))}
$$
Since 
$$
d_{L}(Y,X) = \sup_{\mathcal{F}\in\pmf}\log\frac{i(L_{X}, \mathcal{F})}{i(L_{Y}, \mathcal{F})}
$$
and the set of stable foliations of all pseudo-Anosov maps is dense in $\pmf$ (see e.g. \cite[Theorem 6.19]{FLP}), the lower bound of \eqref{eq.E2est} follows.

Let $\varphi\in\mcg$ be a pseudo-Anosov map with dilatation $\lambda$.
Then, the translation length of $\varphi$ with respect to Thurston's metric $d_{L}$ is known to be $\log\lambda$.
Hence $\tau(\varphi,d_{E,2}) = \log\lambda$ immediately follows from \eqref{eq.E2est2}.
\end{proof}

Now we have a similar result to Proposition \ref{prop.int-est}.
\begin{coro}\label{coro.Ext-est}
\begin{equation}\label{eq.E2est3}
d_{L}(Y,X) - \frac{1}{2}\log\nu(Y) \leq \frac{1}{2}\log({\ext_Y(L_{X})})\leq \dt(X,Y) + \log\left(\frac{~\pi^{3}}{2}|\chi(S)|\right).
\end{equation}
\end{coro}
\begin{proof}
The lower bound is obtained by just putting $b = Y$ in \eqref{eq.E2est}.

The upper bound is given by Theorem \ref{thm.ExHyp} and Proposition \ref{prop.E1}:
\begin{align*}
&\log\frac{\ext_{X}(L_{Y})}{\ext_{Y}(L_{Y})}\leq \sup_{Z\in\T}\log\frac{\ext_{X}(L_{Z})}{\ext_{Y}(L_{Z})} = \dt(X,Y)\\
\Longrightarrow &\log(\ext_{X}(L_{Y}))-\log\left(\frac{~\pi^{3}}{2}|\chi(S)|\right)\leq \dt(X,Y).
\end{align*}
\end{proof}
\begin{rmk}
As the hyperbolic metric is one of the conformal metrics, we always obtain 
$$
\frac{i(X,Y)^{2}}{\mathrm{Area}(Y)}\leq \ext_{Y}(L_{X})
$$
by the definition of the extremal length.
Hence, the lower bound of $\ext_{Y}(L_{X})$ in terms of Thurston's distance and systole of $Y$ follows from Proposition \ref{prop.int-est}.
However, the lower bound in \eqref{eq.E2est3} is better than that from Proposition \ref{prop.int-est}.
\end{rmk}
\subsection{Four compactifications of $\T$ via extremal length}
In Theorem \ref{thm.summary}, we showed that a function $I:\T\times \T$ can possibly give four different compactifications.
\begin{thm}\label{thm.ext-final}
Both the Gardiner-Masur boundary and the Thurston boundary appear naturally as horofunction boundaries via extremal lengths, namely
\begin{enumerate}
\item The horofunction boundary with respect to the Teichm\"uller distance gives the Gardiner-Masur compactification.
\item The horofunction compactification via 
$$\T\ni Z\mapsto\log\ext_{(\cdot)}(L_{Z})-\log\ext_{b}(L_{Z})\in\Lipi(\T,\dt)$$
can be identified with the Thurston compactification.
\end{enumerate}
\end{thm}
\begin{proof}
The first item (1) is due to \cite{LS}.
Let us prove item (2).
The map 
$$
\T\to\Lipi(\T,\dt),\quad Z\mapsto \log\frac{\ext_{(\cdot)}(L_{Z})}{\ext_{b}(L_{Z})}
$$
extends continuously to the $\mathbb{P}\curr$ by \cite{GT}, the image is contained in $\Lipi(\T,\dt)$ by Proposition \ref{prop.E1}.
Recall that the closure of $\T\subset \mathbb{P}\curr$ is identified with the Thurston compactification $\overline\T$ \cite{Bonahon}.
Hence, we have an injective continuous map $\overline\T\to \Lipi(\T,\dt)$ 
from a compact space $\overline\T$ to a metrizable space.
By the standard theory of topology, such a map is a homeomorphism onto its image.
Hence, we have the proof of item (2).
\end{proof}
Although the fact below follows immediately from \ref{coro.Ext-est}, we give an alternative proof using the discussion in Section \ref{sec.horoLip}.
\begin{coro}
Let $\varphi\in\mcg$ be a pseudo-Anosov mapping class with stretch factor $\lambda$.
$$
\lim_{n\to\infty}\frac{1}{n}\log\ext_{X}(\varphi^{n}Y) = \tau_{\dt}(\varphi) = \log\lambda.
$$
\end{coro}
\begin{proof}
By item (2) of Theorem \ref{thm.ext-final}, we see that the assumption of Theorem \ref{thm.ns-trans} is satisfied for $\varphi$.
\end{proof}

\bibliographystyle{alpha} 
\bibliography{references} 
\end{document}